\documentclass{article}

\usepackage{amsmath,amssymb,mathtools,amsthm,bm}
\usepackage{xcolor,graphicx,float,enumerate}

\newcommand{\ee}{\varepsilon}

\newcommand{\cC}{{\mathcal C}}

\def\LL{{\mathcal L}}

\def\HH{\mathcal H}

\newcommand{\RR}{{\mathbb R}}
\newcommand{\Ln}{\mathcal L^n}
\newcommand{\m}{\mu}

\newcommand{\on}{\omega_n}

\DeclareMathOperator{\diver}{div}

\usepackage[utf8]{inputenc}
\DeclareUnicodeCharacter{00A0}{ }

\usepackage{geometry}
\newtheorem{proposition}{Proposition}[section]
\newtheorem{theorem}{Theorem}[section]

\newtheorem{lemma}{Lemma}[section]
\theoremstyle{definition}

\newtheorem{remark}{Remark}[section]

\usepackage[hidelinks]{hyperref}
\usepackage[nameinlink,noabbrev]{cleveref}

\numberwithin{equation}{section}

\begin{document}
\title{Steiner symmetrization for anisotropic quasilinear equations  \\ via partial discretization}

\author{F. Brock\thanks{Institute of Mathematics, University of Rostock} %
	\and J.I. Díaz%
	\thanks{Departamento de Análisis Matemático y Matemática Aplicada, Universidad Complutense de Madrid}%
	~\thanks{Instituto de Matemática Interdisciplinar, Universidad Complutense de Madrid}%
	 \and %
	A. Ferone%
	\thanks{Dipartimento di Matematica e Fisica, Università della Campania ``Luigi Vanvitelli"} %
	\and %
	D. Gómez-Castro\footnotemark[2] \footnotemark[3] %
	\and %
	A. Mercaldo%
	\thanks{Dipartimento di Matematica e Applicazioni ``R. Caccioppoli", Università di Napoli Federico II}
	}
\maketitle

\begin{abstract}
	In this paper we obtain comparison results for the quasilinear equation $-\Delta_{p,x} u - u_{yy} = f$ with homogeneous Dirichlet boundary conditions by Steiner rearrangement in variable $x$, thus solving a long open problem. In fact, we study a broader class of anisotropic problems. Our approach is based on a finite-differences discretization in $y$, and the proof of a comparison principle for the discrete version of the auxiliary problem $A U - U_{yy} \le \int_0^s f$, where $AU = (n\on^{1/n}s^{1/n'} )^p (- U_{ss})^{p-1}$. We show that this operator is T-accretive in $L^\infty$. We extend our results for $-\Delta_{p,x}$ to general operators of the form $-\diver (a(|\nabla_x u|) \nabla_x u)$ where $a$ is non-decreasing and behaves like $| \cdot |^{p-2}$ at infinity.
\end{abstract}

\section*{Introduction}
The aim of this paper is to extend the result on \cite{ADTL} to  a class of  nonlinear  elliptic  problems   whose prototype is  
\begin{equation}
	\label{eq:motivation}
	\begin{cases}
		-\Delta_{p,x}u-u_{yy} =f &  \mbox{in }\Omega\\
		u=0 &\mbox{on } \partial \Omega\,.
	\end{cases}
\end{equation}
For the sake of simplicity, in this paper we focus on the case
\begin{equation}
 \Omega =\Omega_1 \times \Omega_2\,, \qquad \Omega_1 \subset \mathbb R^n \text{ open, bounded of class } \mathcal C^2 \qquad \text{ and } \qquad  \Omega_2 = (0,1).
\end{equation}
We denote $-\Delta_{p,x} u = -\diver (|\nabla_x u|^{p-2} \nabla_x u)$ where $\nabla_x u = (\frac{\partial u}{\partial x_1}, \cdots, \frac{\partial u}{\partial x_n})$. 
Such problem can be regarded as anisotropic since the growth of the operator with respect to the partial derivatives of $u$ in $x$ and $y$ is governed by different powers (see, e.g. \cite[Chapter 1, Section 4.2]{diaz+antontsev2012energy+methods}). \\

Very often, in many relevant applications, the materials and phenomena have an important \textit{anisotropy}, presenting different properties in different directions, in contrast to the more usual isotropy property. 
Anisotropy leads to mathematical models presenting peculiar constitutive
laws which corresponds to material's physical or mechanical properties with different behaviour according the directions. The spectrum of fields in which such situations arise is very wide: Computer Science (e.g. Image processing), Physics (e.g. Atmospheric Radiative Transfer), Chemistry (e.g. Materials Science \cite{newnham2005materials}), Geophysics and Geology, Engineering (e.g. wastewater reactors \cite{Diaz2016}) or Neuroscience. 

Commonly, this anisotropy comes as a diffusion operator $D$
which is described by three principal different directional coefficients. 
This is the case, for example, of thermal and electrical conductivity in heterogeneous media (e.g. for the diffusion of Ni in Olivine) or of study of cristal (see \cite{wulff1901frage,fonseca1991wulff}). 
Moreover, many of the recent innovations in the field of electroceramics have exploited the anisotropy of nonlinearities modelling different material properties such as electric field, mechanical stress or temperature (see, e.g. \cite[Chapter 15]{newnham2005materials}).
We also mention here that it is well-known that homogenisation techniques generate anisotropic diffusion limit problems (see, e.g., \cite{Lions2005,zhikov1999homogenization}, and its references).

From the mathematical point of view, it is useful to get some relations between the solution of a nonlinear anisotropic boundary value problem and the solution of a similar problem, posed in a simpler geometry, for which we can compute explicit information. 
In order to obtain as much information as possible, we would like for this second problem to be as similar as possible to the original one.
This philosophy was developed in the framework of linear second order reaction-diffusion problems by means of the so-called Steiner rearrangement since the nineties of the past century (see, e.g. \cite{alvino+trombetti+diaz1992steiner,ADTL,ALTparabolic}). 
This approach has proved specially useful for the
study of the qualitative behaviour of solutions (see, e.g. \cite[Chapter 1, Section 4.2]{diaz+antontsev2012energy+methods}). 
Nevertheless, the class of operators for which  Steiner rearrangement has been applied is limited, due to technical difficulties. The question of whether it can be applied to non-linear second order problems ``split'' as operators in $x$ and $y$, such as \eqref{eq:motivation}, has been open for 15 years. A related question (but developed with different techniques) concerns anisotropic symmetrisation (see, for example, \cite{AFLTconvexsimmetrization,vanSchaftingen2006anisotropic}).\\

The main goal of this paper is to present new ideas when considering such type of anisotropic nonlinear diffusion operator, trying to extend the main result on \cite{ADTL} to this nonlinear framework. For the moment being, the peculiar formulation we will consider will use the crucial assumption that the diffusion is linear in at least one direction, nevertheless we hope that the ideas could be adapted also without this technical requirement. \\

In the seventies, G. Talenti developed symmetrization techniques in pioonering papers \cite{Talenti1976,Talenti1977} (see also \cite{Weinberger,Mazja}) that allow to obtain a priori estimates on problems of the form
\begin{equation*}
	\begin{cases}
		-\diver( a(|\nabla u|) \nabla u )=f &\mbox{in }   \Omega \\
		u=0 &\mbox{on } \partial\Omega\,.
	\end{cases}
\end{equation*}
where $\Omega$ is an open subset of $\RR^n$ and the function $a$ is monotone satisfying the ellipticity condition
$$
a(\xi)\xi\ge |\xi|^p\,, \qquad \forall \xi \in \RR^n.
$$
Estimates of $u$ in $L^p$ norms  can be derived by estimating the same norm of the solution $v$ to the spherically symmetric problem
\begin{equation}\label{intp}
	\begin{cases}
		-\Delta_{p}v=f^\star  &\mbox{in }   \Omega^\star \\
		v=0 &\mbox{on } \partial\Omega^\star\,.	\end{cases}
\end{equation}
where $\Omega^\star$ is the ball of $\RR^n$ centered at zero having the same Lebesgue measure of $\Omega$, $f^\star$ is the  Schwarz rearrangement of $f$, that is the the spherically symmetric function, decreasing with respect to $|x|$, whose level sets $\{x\in \Omega:|f^\star(x)|>t  \}$ have the same measure of the corresponding level sets of $f$, $\{x\in \Omega:|f(x)|>t  \}$. We will introduce it formally below. 

 G. Talenti's approach has been successfully extended in various directions such as, for example, operators having lower order terms, degenerate operators, parabolic equations, different boundary value problems (see, \cite{AlvinoTrombettiricerche,ALTparabolic,Diaz1991,Talenti1985,TalentiArt} and references therein). \\

 To state the main results of this paper we need to introduce some definitions. Given a bounded set $\omega \subset \RR^n$, we define its Schwarz symmetrisation, $\omega^\star$, as the unique ball centred at $0$ such that $\mathcal L^n (\omega^\star)=\mathcal L^n (\omega)$ where  $\LL^{n}(\omega)$ denotes the $n$-dimensional Lebesgue measure of the set $\omega$.
For an open bounded set $\Omega \subset \RR^N\equiv\RR^n\times\RR^m$, 
for any $y\in\RR^m$, the $y$-section of $\Omega$ is denoted by $\Omega_y$  and defined by 
$
\Omega_y :=\{x\in\RR^n: (x,y)\in\Omega\}
$
and we define the Steiner-symmetrised version of $\Omega$ as 
\begin{equation*}
	\Omega^\# = \bigcup_{y \in \RR^m} ( \Omega_y )^\star \times \{y\}.
\end{equation*} 
When it does not lead to confusion we use the notation $|\omega|$ for the Lebesgue measure of $\omega$ of adequate dimension.

If $u$ is a function defined in  $\Omega \subset \RR^N\equiv\RR^n\times\RR^m$, 
for any $y\in\RR^m$, we consider the function 
\begin{equation}\label{steinf}
x\in \Omega_y\mapsto  u(x,y)\in \RR\,.
\end{equation}
The distribution function and  the decreasing rearrangement of this function \eqref{steinf} are the so-called {\it distribution function (in codimension $n$)} of $u$ and its   {\it decreasing rearrangement (in codimension $n$)} respectively, i.e. 
$$
\m_u(t,y)=\Ln\left (\{x\in\Omega_y :u(x,y)>t\}\right ), \qquad (t,y)\in[0,+\infty )\times\RR^m\,
$$
and
$$
u^*(s,y)=\sup\{t\ge 0: \m_u (t,y)>s\}\, , \quad (s,y)\in \Omega_y^*\times\RR^m\, ,
$$
respectively, where $\omega_n$ is the measure of the unit ball of $\RR^n$.  We define the {\it Steiner symmetrised version of $u$} as
\begin{equation*}
	u^\# (x,y) = u^* (\on |x|^n, y), \qquad (x,y) \in \Omega^\#.
\end{equation*}
Notice that it is spherically symmetric in $x$, and radially non-increasing in this variable. When there is no $y$ variable (i.e. $m = 0$), this is called the {\it Schwarz symmetrised version of $u$}.
 
G. Talenti developed the theory with no $y$ variable ($m = 0$). The Steiner symmetrisation ($m > 0$) is studied in \cite{alvino+trombetti+diaz1992steiner,ADTL,ALTparabolic, bandlekawhol} for the case of linear elliptic operators. The type of comparison results that can be found in the literature are
$$
u^*(s)\le v^*(s)\, , \quad \text{ when $m=0$,}
$$
or 
$$
\int_0^s u^*(\sigma,y )\,d\sigma \le \int_0^s v^*(\sigma,y )\,d\sigma \,, 
\quad \mbox{ when $m>0$,}
$$
for a.e.\ $s\in (0,|\Omega|)$ and a.e. in $y \in \Omega_2$, which easily imply the a priori estimate on $u$  in $L^p$ or Orlicz norms.
Similar results for $m > 0$ have also been proven in a more recent paper \cite{BCFM} by using a simpler approach; Neumann boundary value problems have been studied in \cite{FeroneMercaldo} (see also \cite{chiacchio}). The implications of these kinds of mass comparison are deep. For example, for any $q \ge 1$ it allow us to have estimates on the $L^q$ norm of $u$
\begin{equation*}
	\int_{\Omega_1 \times\Omega_2} |u (x,y)|^q dx dy \le \int_{\Omega_1^\star \times\Omega_2} |v (x,y)|^q dx dy.
\end{equation*}
This is useful because the radially symmetric problem can often be estimated easily by direct techniques. 

\section{Main result and structure of the paper}

We propose a new approach, that covers a wider class of problems of the form
\begin{equation}
	\tag{P}
	\label{eq:P}
	\begin{cases}
	-\diver_x \Big( a(|\nabla_x u| ) \nabla_x u \Big) -u_{yy}=f &\mbox{in } \Omega_1\times\Omega_2\\
	u=0 &\mbox{on } \partial(\Omega_1\times\Omega_2)\,.
\end{cases}
\end{equation}
where
\begin{align}
	\tag{H$_1$}
	\label{eq:a general}
	a : (0,+\infty) \to (0,+\infty) \text{ such that } \beta(t) = \begin{dcases}
	a(t) t & t > 0 \\
	0 & t = 0
	\end{dcases} \text{ is  non-decreasing}
\end{align}
and, setting $A(t) = t \beta(t) = t^2 a(t)$, for some $C_1,C_2 > 0$ and $p > 1$ we have
\begin{align}
	\tag{H$_2$}
	\label{eq:beta p estimates}
	C_1  ( t^p - 1) \le   A(t), \qquad \beta(t)  \le C_2 (t^{p-1} + 1), \qquad \text{and} \qquad A(t) \text{ is convex}.
\end{align}

Notice that the solution of \eqref{eq:P} is a minimiser of the energy
\begin{equation}
\label{eq:energy}
J(u) = \int_{\Omega} \Big( B(|\nabla_x u|) + |\nabla_y u|^2 - f u \Big) dx\, dy
\end{equation}
where
\begin{equation}
\label{eq:B and b}
B(t) = \int_0^t \beta(s) ds.
\end{equation}
Since $\beta$ is non-decreasing, $u\mapsto \int_{\Omega} B(|\nabla_x u|)$ is convex.  On the other hand, $\int_{\Omega}|\nabla_y u|^2$ is strictly convex in $L^2(\Omega_1; H_0^1(\Omega_2))$, we deduce that $J$ is strictly convex. Therefore, a unique weak solution of \eqref{eq:P} exists. The natural space of solutions for this problem is precisely
\begin{equation*}
	X^p(\Omega) = \{  u \in W^{1,1}_0 (\Omega) : |\nabla_x u| \in L^p (\Omega), |\nabla_y u| \in L^2 (\Omega) \}.
\end{equation*}
If one does not introduce condition the power bounds in \eqref{eq:beta p estimates} then the energy must be found in the Orliz class $B(|\nabla_x u|) \in L^1 (\Omega)$ (see, e.g., \cite{Talenti1977}).\\

Our aim is to prove the following result
\begin{theorem}\label{thm:main}
Let $a$ satisfy \eqref{eq:a general} and \eqref{eq:beta p estimates}, $0\le f \in L^{ \max \{2,p\} } (\Omega)$, $u \in X^p (\Omega)$ be the weak solution of the  problem \eqref{eq:P} and $v\in X^p (\Omega^\#)$ be the weak solution of the symmetrised problem
\begin{equation}
\tag{P$^\#$}
\label{eq:P sharp}
\begin{cases}
	-\diver_x( a(|\nabla_x v|) \nabla_x v  )-v_{yy}=f^\# &\mbox{in } \Omega^\# \\
	v=0 &\mbox{on } \partial \Omega^\#\,.
\end{cases}
\end{equation}
Then, for the decreasing rearrangements $u^*$ and $v^*$ we have the following mass comparison:
\begin{equation}
	\label{aim}
	\int_0^s u^*(\sigma,y)d\sigma\le \int_0^s v^*(\sigma,y)d\sigma\,, \qquad \mbox{for all } s\in [0,|\Omega_1|] \mbox{ and for a.e. } y\in\Omega_2\,.
\end{equation}
\end{theorem}
Since it is  known that the rearrangement is continuous from $L^1 (\Omega) \to L^1 (\Omega_1^* \times \Omega_2)$, we preserve \eqref{aim} by using approximate problems such that the solutions converge to $u$ and $v$ at least in $L^1(\Omega)$ and $L^1 (\Omega^\#)$. \\

The structure of the proof is as follows. We consider first, in \Cref{sec:smooth elliptic case}, the case for a smooth function $\beta  \equiv \beta _{\varepsilon}$ satisfying elliptic conditions, i.e.
\begin{equation}
\tag{H$_\ee$}
\label{eq:beta smooth}
\beta \in \mathcal C^1( (0,+\infty)) \quad  \text{with} \quad \beta(0) = 0, \quad \beta' : [0,+\infty) \to  \left  [\ee, \frac 1 \ee \right ] \text{ for some } \ee > 0,  \quad \text{ and }
\quad \beta (t) t \text{ is convex}.
\end{equation}
Following more or less classical arguments, we show that the solution of \eqref{eq:P} under \eqref{eq:a general} and \eqref{eq:beta p estimates} can be approximated by problems with $ a(t)\equiv a_{\varepsilon}  (t) = \beta_{\ee}(t)/t$ and $\beta_\ee$ satisfying \eqref{eq:beta smooth}. 
\\ 

Under assumptions \eqref{eq:beta smooth} we discretize in the $y$ derivative, to obtain a family of problems 
\begin{equation}
	\tag{P$_h$}
	\label{eq:Ph}
	\begin{dcases}
		-\diver_x\Big(a(|\nabla_x u_{j}|)\nabla_x u_j\Big) - \frac{u_{j+1} - 2 u_j + u_{j-1}}{h^2} = f_j & \text{in } \  \Omega _1 ,
\\
		u_j = 0 & \text{on } \partial \Omega_1 , \quad j=1, \cdots, N , \\
		u_0 = u_{N+1} \equiv  0 & \text{in } \  \Omega_1,
	\end{dcases}
\end{equation}
where 
\begin{equation} 
	(N+1)h = 1. 
\end{equation}
We will show that we recover solutions of \eqref{eq:P} as $h \to 0$. 
Due to \eqref{eq:beta smooth} we can use $p=2$ in the study of \eqref{eq:Ph}. 
Analogously to \eqref{eq:P}, we will show in \Cref{sec:existence and uniqueness of Ph} that the solution of \eqref{eq:Ph} is a minimiser of
\begin{equation}
\label{eq:Ph energy}
J_h (\mathbf u) = \sum_{i=1}^N \int_{\Omega_1}  B(|\nabla_x u_j|) \, dx
+ \sum_{i=0}^N \int_{\Omega_1} \left(\frac{u_{j+1}-u_j}{h}\right)^2 \, dx  
- \sum_{i=1}^N \int_{\Omega_1} f_j u_j \, dx\, ,
\end{equation}
where
\begin{equation}
\mathbf u=(u_j)\in	X_N (\Omega_1) = \{ \mathbf u \in H^1_0(\Omega_1)^{N+2} : u_0 = u_{N+1} = 0  \}  = \{0_{H_0^1}\} \times H^1_0(\Omega_1)^N \times \{0_{H_0^1}\} .
\end{equation}

\noindent We will explain the construction of this energy functional, and provide an existence and uniqueness result for the case of smooth $\beta$.

\noindent  Since we want to apply rearrangement properties of smooth functions, we devote some time to the regularity of the solution of this system. In order to study regularity of solutions of each $u_j$, we can move the discrete Laplacian to the right hand side, and recover a problem of the form
\begin{equation}
\label{eq:Cianchi-Mazya problem}
\begin{dcases} 
-\diver (a(|\nabla w|) \nabla w) = g & \textrm{in } \Omega_1 , \\
w = 0 & \textrm{on } \partial \Omega_1. 
\end{dcases}
\end{equation}
Existence, uniqueness and regularity for this problem has been studied in a series of recent papers by Cianchi and Maz'ya (e.g. \cite{Cianchi2014,Cianchi2018}) under the assumption that $a$ is smooth and has some type of coercitivity
\begin{equation}
\tag{CM}
\label{eq:Cianchi}
a \in \mathcal C^1(0,+\infty), \qquad
i_a = \inf_{t>0} \frac{ t a'(t) }{ a(t) } > -1  , \qquad 	s_a = \sup_{t>0} \frac{ t a'(t) }{ a(t) }    < \infty.
\end{equation}
They show that \eqref{eq:Cianchi} is sufficient to imply \eqref{eq:a general} and \eqref{eq:beta p estimates}. 
We will apply their results for  a regularisation of $a$, and we will pass to the limit for general operators. Notice that \eqref{eq:Cianchi} holds directly for the case of the $p$-Laplace operator.\\

In this setting we will easily prove, in \Cref{sec:rearrangement of Ph}, comparison results in a very classical manner: we rearrange each equation and apply a comparison argument for the system \eqref{eq:Ph}. We define
\begin{equation*} 
U_j (s) = \int_0^s u_j^* (\sigma) d \sigma, 
\qquad 
V_j(s) = \int_0^s v_j^*(\sigma) d\sigma,
\qquad
F_j(s) = \int_0^s f_j(\sigma)d\sigma
\end{equation*}  
and, for $j \in \{1, \cdots, N\}$, we have that $U_j$ is a weak solution of 
\begin{equation}
	\label{eq:Ph star subsol}
	\tag{$\underline {\textrm{P}_h^*}$}
	\begin{dcases}
	\kappa_n(s) \beta \left (-\kappa_n(s) \,\frac{d^2 U_j}{d s^2}\right) -\frac{U_{j+1} - 2U_j + U_{j-1}}{h^2}\le F_j & \text{ in } \Omega_1^*, \\
	\frac{dU_j}{ds} (|\Omega_1|) =	U_j(0) = 0,
	\end{dcases}
\end{equation}
where
\begin{equation}
	\kappa_n(s) = n\on^{1/n}s^{1/n'}
\end{equation}
in the sense that $U_j ,\ \kappa_n(s) \,\frac{d^2 U_j}{d s^2} \in L^\infty (\Omega_1^*)$ and the equation is satisfied almost everywhere. 
We will show that $V_j$ solves the same problem, except that the above differential inequalities become  equalities (see \eqref{eq:Ph star} below). 
Due to regularity of $u_j$ that we will prove, we recover some regularity of $U_j$. This regularity is sufficient to apply accretivity results for \eqref{eq:Ph star subsol} in $L^\infty$. 
We devote \Cref{sec:comparison principle} to prove this result, from which we deduce $U_j \le V_j$ for every $j$. 
This is precisely the mass comparison we sought, at least for \eqref{eq:Ph}. We devote \Cref{sec:convergence Ph to P} to showing that we can pass to the limit as $h \to 0$, and recover solutions of the original problem. In \Cref{sec:comparison in smooth case}, we pass to the limit as $h \to 0$ in the comparison, thus proving the comparison under assumptions \eqref{eq:beta smooth}. Then, in \Cref{sec:proof of comparison general} we use an approximation argument to prove the main result in the general setting.

\section{Smooth elliptic case}
\label{sec:smooth elliptic case}

In this part, we assume \eqref{eq:beta smooth}.
This immediately implies that $\ee t \le  \beta(t) \le t / \varepsilon$, so \eqref{eq:beta p estimates} holds with $p = 2$. 
The aim of this section is to prove
\begin{theorem}\label{thm:main smooth}
	Let $a$ satisfy \eqref{eq:a general} and \eqref{eq:beta smooth}, $0\le f \in \mathcal C^\infty_c (\Omega)$ and let $u \in H_0^1 (\Omega)$ be the weak solution of the  problem \eqref{eq:P} and $v\in H_0^1  (\Omega^\#)$ be the weak solution of the symmetrised problem \eqref{eq:P sharp}.
	Then, we have the following mass comparison:
	\begin{equation}
	\int_0^s u^*(\sigma,y)ds\le \int_0^s v^*(\sigma,y)dy\,, \qquad \mbox{for all } s\in [0,|\Omega_1|] \mbox{ and for a.e. } y\in\Omega_2\,.
	\end{equation}
\end{theorem}

Until \Cref{sec:convergence Ph to P}, variable $y$ is not present, and so we denote $\nabla_x$ simply by $\nabla$.

\subsection{Existence, uniqueness and regularity of solutions of the discrete problem \eqref{eq:Ph}}
\label{sec:existence and uniqueness of Ph}
We say that a function $\mathbf u\in X_N(\Omega_1) $ is a {\it weak solution} of \eqref{eq:Ph} if
\begin{align*}
\int_{\Omega_1} a(\nabla u_j  )\nabla u_j& \cdot \nabla \varphi_j- \int_{\Omega_1} \frac{u_{j+1} - 2u_{j} + u_{j-1}}{h^2} \varphi_j = \int_{\Omega}f_j\varphi_j, \quad \forall j \in \{1,\cdots, N\} \ \ \forall \bm \varphi \in X_N(\Omega_1).
\end{align*}

Notice that, for $\mathbf u, \bm \varphi \in X_N(\Omega_1)$ we have that
\begin{equation}
	\label{eq:discrete u_yy}
	\sum_{j=1}^N \frac{-u_{j+1} + 2 u_{j} - u_{j-1}}{h^2} \varphi_j = \sum_{j=0}^N \frac{u_{j+1} - u_{j}}{h}\frac{\varphi_{j+1} - \varphi_{j}}{h} = \sum_{j=1}^N u_j \frac{-\varphi_{j+1} + 2 \varphi_{j} - \varphi_{j-1}}{h^2} .
\end{equation}
Hence, it is easy to see that we can write equivalently the weak formulations
\begin{equation}
	\label{eq:Ph weak vector}
	\int_{\Omega_1} \sum_{j=1}^N a(\nabla u_j  )\nabla u_j \cdot \nabla \varphi_j 
	+ \int_{\Omega_1}  \sum_{j=0}^N \frac{u_{j+1} - u_{j}}{h}\frac{\varphi_{j+1} - \varphi_{j}}{h} 
	= \int_{\Omega_1} \sum_{j=1}^N f_j  \varphi_j,
\end{equation}
and
\begin{equation}
		\label{eq:Ph very weak vector}
	\int_{\Omega_1} \sum_{j=1}^N a(\nabla u_j  )\nabla u_j \cdot \nabla \varphi_j 
	+  \int_{\Omega_1}  \sum_{j=1}^N u_j \frac{-\varphi_{j+1} + 2 \varphi_{j} - \varphi_{j-1}}{h^2} 
	= \int_{\Omega_1} \sum_{j=1}^N f_j  \varphi_j.
\end{equation}

\begin{proposition}
	\label{prop:existence of Ph}
	Assume \eqref{eq:beta smooth}, and let $\mathbf f = (f_j)  \in L^2 (\Omega)^N $ where $ f_j \ge 0 $. Then, there exists a unique $\mathbf u = (u_j) \in X_N(\Omega_1)$ where $u_j\ge 0$ satisfy \eqref{eq:Ph weak vector}. It also satisfies \eqref{eq:Ph very weak vector} and is the global minimiser in $X_N(\Omega_1)$ of $J_h$ given by \eqref{eq:Ph energy}.
\end{proposition}
\begin{proof} Since $B'' = \beta' \ge  \ee$, 
	$B$ is strictly convex, quadratic and bounded from below. Hence $J_h$ has a unique minimiser. Applying \eqref{eq:discrete u_yy} and reproducing the proof we deduce that the Euler-Lagrange equations for $J_h$ are precisely \eqref{eq:Ph}. To check that $u_j \ne 0$ we use $\varphi_j = (u_j)_-$ as a test function, to deduce $u_- = 0$. 
\end{proof}

One of the advantages of having discretised in $y$, is that for every $j$ we move the right hand side
\begin{equation*}
		-\diver \Big( a( |\nabla u_j| ) |\nabla u_j| \Big) = F_j = f_j + \frac{u_{j+1} - 2u_j + u_{j-1}}{h} \qquad \textrm{ in } \Omega_1.
\end{equation*}
This a system of equations. It is called diagonal since the $j$-th equation only includes the gradient of $u_j$. The coupling $F_j$ is linear in $\mathbf u$. It is proven in \cite[Theorem 2]{Meier1982} that, if $\mathbf f \in L^\infty(\Omega_1)^{N}$ then $\mathbf u \in L^\infty(\Omega_1)^N$.\\

From the series of papers by Cianchi and Maz'ya we recover some regularity results. In particular, in \cite{Cianchi2014} the authors prove that, 
	for the solution $w$ of \eqref{eq:Cianchi-Mazya problem}, we have 
\begin{equation}
	 \| \nabla w \|_{L^\infty (\Omega_1) } \le C \beta^{-1} \left(  \| g \|_{L^{n,1}(\Omega_1) }  \right),
\end{equation}
and in \cite{Cianchi2018} they prove 
\begin{equation}
	a(|\nabla w|)  \nabla w  \in W^{1,2}(\Omega_1) \iff g \in L^2 (\Omega_1). 
\end{equation}
Notice that
\begin{equation*}
	\frac{ta'(t)}{a(t)} = \frac{t \beta'(t)}{\beta (t)} - 1.
\end{equation*}
Since $\ee \le \beta' \le \frac 1 \ee$ and $\beta (0) = 0$,  we have that $\ee t \le \beta(t) \le t/\ee$ and hence $1 \le \frac{t \beta'(t)}{\beta} \le \frac 1 {\ee^2}$. Therefore, \eqref{eq:beta smooth} implies \eqref{eq:Cianchi}.
Applying these two results, and the fact that $u_j \in L^2(\Omega_1)$ by the minimisation argument, we have
\begin{theorem}
	Let $\mathbf f \in L^\infty(\Omega_1)^N$. Then, the unique weak solution of \eqref{eq:Ph} is in $ W^{1,\infty}_0(\Omega_1)^{N+2}$ and
	\begin{equation}
		\label{eq:regularity a(nabla u) nabla u}
		a(|\nabla u_j |) \nabla u_j \in H^1  (\Omega_1).	
	\end{equation}
\end{theorem}

\subsection{Rearrangement of \eqref{eq:Ph} to a problem \eqref{eq:Ph star}}
\label{sec:rearrangement of Ph}
Our aim is to compare \eqref{eq:Ph} with its rearranged problem:
\begin{equation}
	\tag{P$_h^\#$}
	\label{eq:Ph sharp}
	\begin{dcases}
		-\diver_x\Big(a(|\nabla v_{j}|)\nabla v_j\Big) - \frac{v_{j+1} - 2 v_j + v_{j-1}}{h^2} = f_j^\star & \text{in } \Omega_1^\star, \\
		v_j = 0 & \text{on } \partial \Omega_1^\star , \ j= 1, \cdots, N ,\\
		v_0 = v_{N+1} = 0 & \text{in } \Omega_1^\star .
\end{dcases}
\end{equation}
Arguing as before, it has a unique solution $\mathbf v \in X_N(\Omega_1^\star)$. 
For every  $q\in [1,\infty]$ we denote by $q':=\frac q {q-1}$ its conjugate exponent.
\begin{proposition}
	\label{prop:Ph rearrangement} 
	Let $\mathbf f \in \mathcal C_c (\Omega_1)^N$ and let $\mathbf u \in X_N(\Omega_1)$ and $\mathbf v \in X_N(\Omega_1^\star)$ be the unique solutions of \eqref{eq:Ph} and \eqref{eq:Ph sharp} respectively.
	Define, for every $j \in \{0, \cdots, N+1\}$ 
	\begin{align*}
	U_j(s)=\int_0^s u_j^*(\sigma)\,d\sigma,
	\qquad 
	V_j(s)=\int_0^s v_j^*(\sigma)\,d\sigma,
	\qquad  	
	F_j(s) = \int_0^s f_j^*(\sigma)d\sigma.	
	\end{align*}
	Then, for every $j \in \{1, \cdots, N\}$, $U_j$ and $V_j$ are in $\cC(\overline {\Omega_1^*})$ and satisfy
	\begin{align}
		\label{eq:integrability of U_{ss}}
		\kappa_n(s) \,\frac{d^2 U_j}{d s^2},\quad & \kappa_n(s)  \,\frac{d^2 V_j}{d s^2} \in L^\infty(\Omega_1^*). 
	\end{align} 
	Moreover $\mathbf U = (U_j)$ is a solution of \eqref{eq:Ph star subsol} and $\mathbf V = (V_j)$ is a solution of
	\begin{equation}
	\tag{P$_h^*$}
	\label{eq:Ph star}
	\begin{dcases}
	\kappa_n(s)  \beta \left (-\kappa_n(s)  \,\frac{d^2 V_j}{d s^2}\right) -\frac{V_{j+1} - 2V_j + V_{j-1}}{h^2}= F_j & \text{in } \Omega_1^*, \\
	\frac{dV_j}{ds} (|\Omega_1|) =	V_j(0) = 0 .
	\end{dcases}
	\end{equation}
	Also, $U_0 = U_{N+1} = V_0 = V_{N+1} = 0$. 
\end{proposition}

Before we proceed to the proof, we recall some classical results of rearrangement theory. 

\subsubsection{Some Schwarz rearrangement results}

For the discrete problem there is no $y$ variable, and so we can apply standard results from Schwarz rearrangement. Consider $u : \Omega \to \RR$ non-negative. We define the Schwarz rearrangement
\begin{equation}
	u^\star(x) = u^* (\on |x|^n), \qquad \text{for } x \in \Omega^\star.
\end{equation}

The relation between $u^*$ and $\mu$ is the following
\begin{equation*}
	\mu(u^*(s)) = |\{x \in \Omega : u(x) > u^*(s) \}| \le s\le |\{x \in \Omega : u(x) \ge u^*(s) \}| =  \mu(u^*(s)^-)
\end{equation*}
and equalities hold if and only if $\mu$ is continuous or, equivalently, if $u^*$ has no flat zone. Since $\mu$ is monotone, the set of discontinuities is, at most, countable, hence has measure zero.

The rearrangement of $u$ is constructed so that, 
for any $A \subset \Omega$,
\begin{align*}
	\int_{A} u (x) dx &\le \int_0^{|A|} u^* (\sigma) d\sigma, 
\end{align*} 
and, for a.e. $s \in \Omega^*$
\begin{align*} 
	\int_{u > u^*(s) } u \, dx &= \int_0^{s} u^* (\sigma) d\sigma .
\end{align*}
It is well known that if $u\in W_0^{1,p}(\Omega)$, for some $1\le p\le \infty$, then also $u^\star \in W_0^{1,p}(\Omega^\star)$, and, by the classical P\'olya-Szeg\"o inequality, the $L^p$ norm is preserved while the $W^{1,p}$ norm is reduced (see for example \cite{Baernstein1976,Brock2000,Brock1995,Burchard1997} and  the references therein), in the sense that
\begin{equation}\label{PS}
\int_{\Omega} a(|\nabla u|) |\nabla u|^2 dx \ge \int_{\Omega^\star}a(|\nabla u^\star|) |\nabla u^\star|^2 dx\,.
\end{equation}
The reader may find a discussion on extremals of this kind of inequality in \cite{Burchard2015}.
Inequality \eqref{PS} is a consequence of the classical co-area formula and  of the following  inequalities (see, for example, \cite{Talenti1977})
\begin{equation}\label{PSlevelx}
\int_{u(\cdot)=t}a(|\nabla u|) |\nabla u|\,d\HH^{n-1}\ge\int_{u^\star(\cdot)=t}a(|\nabla u^\star|) |\nabla u^\star|\, d\HH^{n-1} , \qquad \forall t > 0.
\end{equation}
By definition, we easily deduce that
\begin{equation}\label{unigradx}
|\nabla u^\star(x)|=\left [  \kappa_n(s) \left (-\frac{d u^*}{d s}(s,y)\right ) \right ] \Bigg| _{ s=\on |x|^n} \qquad \mbox{ for a.e. } x\in\Omega^\star,
\end{equation}
and then \eqref{PSlevelx} becomes
\begin{align}
\int_{u(\cdot)=t}a(|\nabla u|)&|\nabla u|\,d\HH^{n-1} \notag\\
&\ge a(|\nabla u^\star|)|\nabla u^\star|\HH^{n-1}\left (\left  \{x: u^\star(x)=t\right \}\right )\notag\\
&=\beta(|\nabla u^\star|) |\HH^{n-1}\left (\left  \{x: u^\star(x)=t\right \}\right )\notag\\
&= \kappa_n(s)\beta  \left (- \kappa_n(s)\frac{d u^*}{d s}(s,y)\right )\Bigg |_{s=\mu(t)} .
\label{PSlevx}
\end{align}
\begin{lemma}
	\label{lem:PS}
	Let $u \in W^{1,\infty} (\Omega_1)$. Then, $u^*$ is differentiable a.e. and
	$$
		0 \le -  \kappa_n(s) \frac{du^*}{ds} \in L^\infty(\Omega_1^*).
	$$
	If, furthermore, $a(|\nabla u|) \nabla u \in H^1 (\Omega_1)$ then, for a.e. $s \in \Omega_1^*$, we have
	\begin{align}
	\label{eq:rearrangement div inequality}
	-\int_{u>u^*(s)} \diver\left(a(|\nabla u|)\nabla u\right)\,dx
	\ge   \kappa_n(s) \beta \left (- \kappa_n(s)\frac{d u^*}{d s}(s)\right ). 
	\end{align}
\end{lemma}

\begin{proof}
	We split the proof in several steps.
	\subparagraph{Step 1. $u \in \cC^\infty_c (\Omega)$.} By the divergence theorem, for a.e. $s \in \Omega_1^*$ the outer normal to $\{x: u(x)>u^*(s)\}$ is given by
	$$\nu (x)=-\frac {\nabla u(x)}{|\nabla u(x)|} \qquad \mbox{for } \HH^{n-1}\mbox{-a.e. } x\in\{u =u^*(s)\}\,,$$
	we get
	\begin{align}
	\label{p11} 
	-\int_{u>u^*(s)} \diver_x\left(a(|\nabla u|)\nabla u(x)\right) \,dx
	=\int_{u=u^*(s)} a(|\nabla u|) |\nabla u(x)|\, d\HH^{n-1}.
	\end{align}
	Taking \eqref{PSlevx} into account, we prove the result. 
	
	\subparagraph{Step 2. General case.} 
		 Let $u$ be as in the statement. Since $u$ is Lipschitz continuous and vanishes on the boundary, by \cite{Gehring1961} then $u^\star$ is Lipschitz continuous. In particular $\kappa_n(s) {du^*}/{ds} \in L^\infty(\Omega^*)$. 
		 There exists a sequence $u_k \in \cC^\infty_c (\Omega)$ such that
		 \begin{align*}
		 	u_k &\to u  \qquad  \text{ in } L^1( \Omega ) \\
		 	\nabla u_k & \rightharpoonup \nabla u  \qquad  \text{ in } L^\infty( \Omega )^n \\
		 	a(|\nabla u_k|) \nabla u_k & \rightharpoonup a(|\nabla u|) \nabla u \quad   \text{ in } H^1( \Omega )^n.
		 \end{align*}
		\subparagraph{Step 2a. Convergence of the rearranged term} We prove that%
		 
		, up to a subsequence, for any $0 \le \varphi \in L^\infty (0,|\Omega_1|)$ we have that
		\begin{align}
		\liminf_k \int_0^{|\Omega_1|} &  \kappa_n(s) \beta \left (-\kappa_n(s)\frac{d u_{k}^*}{d s}(s)\right ) \left( -\frac{d u_{k}^*}{d s}(s) \right) \varphi(s) ds \nonumber  \\
		&\ge \int_0^{|\Omega_1|} \kappa_n(s)  \beta \left (-\kappa_n(s)\frac{d u^*}{d s}(s)\right ) \left( - \frac{d u^*}{d s}(s) \right) \varphi(s) ds
					\label{eq:converge Au_k}
		\end{align}		
		It is clear that $\|\kappa_n(s) du_k^* / ds \|_{L^\infty} \le C$, hence, up to a subsequence (still denoted by $u_k$)
		 \begin{equation*}
		 	\kappa_n(s) \frac{du_{k}^*}{ds} \overset \star \rightharpoonup \xi \qquad \text{ in } L^\infty (\Omega^*).
		 \end{equation*}
		 Since $u_k \to u$ in $L^1 (\Omega)$ we have $u_k^* \to u^*$ in $L^1 (\Omega^*)$. Hence, for $\varphi$ such that $\kappa_n(s) \varphi \in W^{1,\infty }(\Omega^*)$ we have
		 \begin{equation*}
		 	\int_{\Omega_1^*} \xi \varphi = \lim_k \int_{\Omega_1^*} \kappa_n(s) \frac{du_{k}^*}{ds} \varphi = -	\lim_k \int_{\Omega_1^*}u_{k}^*  \frac{d}{ds} (\kappa_n(s) \varphi) = - \int_{\Omega_1^*}u ^*  \frac{d}{ds} (\kappa_n(s)\varphi) = \int_{\Omega_1^*}\kappa_n(s)\frac{du^*}{ds} \varphi .
		 \end{equation*}
		 Hence, 
		 \begin{equation*}
		 	\xi = \kappa_n(s)\frac{du^*}{ds}.
		 \end{equation*}
		 
		Fix $ 0\le \varphi \in L^\infty (0,|\Omega_1|)$. Since $A$ is convex and continuous, the map
		\begin{equation*}
			g \mapsto \int_0^{|\Omega_1|} A(g(s)) \varphi(s) ds
		\end{equation*}
		is convex and lower semicontinuous in the topology of $L^p (\Omega)$ for any $p \ge 1$. Therefore, it also weak-lower semicontinuous in $L^p (\Omega)$. Thus,		
		 \begin{align*}
		 	\liminf_k \int_0^{|\Omega_1|}  & \beta \left (-\kappa_n(s)\frac{d u_{k}^*}{d s}(s)\right ) \left( - \kappa_n(s)\frac{d u_{k}^*}{d s}(s) \right) \varphi(s) ds\\
		 	&= 		\liminf_k \int_0^{|\Omega_1|}  A \left (-\kappa_n(s)\frac{d u_{k}^*}{d s}(s)\right ) \varphi(s) ds\\
		 	& \ge 	 \int_0^{|\Omega_1|}  A \left (-\kappa_n(s)\frac{d u^*}{d s}(s)\right ) \varphi(s) ds\\
		 	&=	\int_0^{|\Omega_1|}   \beta \left (-\kappa_n(s)\frac{d u^*}{d s}(s)\right ) \left( - \kappa_n(s)\frac{d u^*}{d s}(s) \right) \varphi(s) ds
		 \end{align*}

		 \subparagraph{Step 2b. Convergence of the divergence term} Let us prove that
		 \begin{equation}
		 \label{eq:convergence int a(u_k)u_k}
		 -\int_{\{u_{k}>u_{k}^*(\cdot)\}}  \diver\Big (a(|\nabla u_{k}|)\nabla u_{k}\Big) \, dx  \longrightarrow  -\int_{\{u>u^*(\cdot)\}}  \diver\Big (a(|\nabla u|)\nabla u\Big) \, dx, \qquad \text { in } L^1 (\Omega^*).
		 \end{equation}
		 Consider the map
		 \begin{equation*}
		 s \in \Omega_1^* \mapsto F_k (s) = -\int_{\{u_{k}>u_{k}^*(s)\}}  \diver\Big (a(|\nabla u_{k}|)\nabla u_{k}\Big) \, dx = -\int _\Omega   \diver\Big (a(|\nabla u_{k}|)\nabla u_{k}\Big) \, \chi_{\{u_{k}>u_{k}^*(s)\}} \, dx .
		 \end{equation*}
		 We have that $\diver\Big (a(|\nabla u_{k}|)\nabla u_{k}\Big)$ converges weakly in $L^2$. Let us prove that, for a.e. $s \in \Omega^*$
		 \begin{equation}
		 	\label{eq:measure of level set u_k}
		 	\chi_{\{u_{k}>u_{k}^*(s)\}}   \longrightarrow \chi_{\{u >u^*(s)\}}  \text{ in } L^2 (\Omega)  .
		 \end{equation} 
		 First, let us prove the convergence a.e. $x \in \Omega$:
		 if, $s$ is such that 
		 \begin{equation}
		 	\label{eq:ukstar pointwise}
		 	\lim_k u_k^*(s) = u^* (s) 
		 \end{equation}
		 then
		 \begin{equation*}
		 	\left \{ x \in \Omega : \lim _k \chi_{\{u_{k}>u_{k}^*(s)\}} (x)  \ne \chi_{\{u >u^*(s)\}} (x)    \right \} \subset  \{ x \in \Omega: u(x) \ne u^*(s) \}.
		 \end{equation*}
		 Indeed, let $s \in \Omega^*$ and $x \in \Omega$ be such that $u(x) < u^*(s)$. Take $\varepsilon = (u^*(s) - u(x)) / 4$. For $k \ge k_\ee$ large enough $|u^*_k(s) - u^*(s)| \le \varepsilon$ and (since $u_k$ converges in $\cC (\overline \Omega)$),  $|u_k(x) - u(x)| \le \varepsilon$. But then $u_k (x) < u_k^*(s)$. Hence $\chi_{\{u_{k}>u_{k}^*(s)\}} (x) = \chi_{\{u >u^*(s)\}} (x)$. The same holds for the limit. We can repeat the same argument if $u(x) > u^*(s)$. \\ 
		 Since $u_k^* \to u^*$ in $L^1 (\Omega^*)$, up to a subsequence, $u_k^* \to u^*$ a.e. Hence, \eqref{eq:ukstar pointwise} holds a.e. On the other hand, $$
		 \LL^n \{ x \in \Omega: u(x) \ne u^*(s) \} = 
		 \mu(u^*(s)^-) - \mu (u^*(s)).$$ 
		 Since $u^*$ and $\mu$ are monotone functions, the set of $s$ such that $\mu(u^*(s))$ is discontinuous at $s$ is countable. Hence, the set of $s$ such that \eqref{eq:measure of level set u_k} does not hold has measure $0$. \\
		  Since the sequence is pointwise bounded by $1$, due the Dominated Convergence Theorem we have \eqref{eq:measure of level set u_k}. \\ 
		 Hence, as $k\rightarrow +\infty,$ 
		 \begin{equation*}
			 	F_k(s) = -\int_{\Omega}  \diver\Big (a(|\nabla u_{k}|)\nabla u_{k}\Big) \chi_{\{u_{k}>u_{k}^*(s)\}}\,dx \longrightarrow 	-\int_{\Omega}  \diver\Big (a(|\nabla u|)\nabla u\Big) \chi_{\{u>u^*(s)\}}\,dx, \qquad \text{a.e. } s \in \Omega^*.
		 \end{equation*}
		 It is clear that
		 \begin{equation*}
		 	|F_k(s)| \le \int_{\Omega}  \left | \diver\Big (a(|\nabla u_{k}|)\nabla u_{k}\Big) \right| dx \le C.
		 \end{equation*}
		 Since we have the pointwise limit, due to the dominated convergence theorem, we recover \eqref{eq:convergence int a(u_k)u_k}.

		 \subparagraph{Step 2c. Comparison of the limits}We apply Step 1 to this final subsequence. 
		 
		 We multiply both sides by $-\frac{du^*}{ds} \varphi(s)$, integrate in $s$ and pass to the limit to deduce that
		 \begin{align*}
		 	\int_0^{|\Omega_1|}& \left \{ - \int_{\{u>u^*(s)\}}   \diver\Big (a(|\nabla u|)\nabla u\Big) \, dx \right \}  \left( - \frac{du^*}{ds} \right)  \varphi (s) d s \\ 
		 	& \ge   \int_0^{|\Omega_1|}  \kappa_n(s) \beta \left (- \kappa_n(s)\frac{d u^*}{d s}(s)\right )  \left( - \frac{du^*}{ds} \right)  \varphi (s) d s\,.
		 \end{align*}
		 Since this holds for any $\varphi$, we have that for a.e. $s \in [0,|\Omega_1|]$
		 \begin{equation*}
		 	 \left \{ - \int_{\{u>u^*(s)\}}  \diver\Big (a(|\nabla u|)\nabla u\Big) \, dx \right \}  \left( - \frac{du^*}{ds} \right)  \ge  \kappa_n(s) \beta \left (- \kappa_n(s)\frac{d u^*}{d s}(s)\right )  \left( - \frac{du^*}{ds} \right)  .
		 \end{equation*}
		 Taking into account \eqref{p11}  we have that
		 \begin{equation*}
		 	- \int_{\{u>u^*(s)\}}  \diver\Big (a(|\nabla u|)\nabla u\Big) = - \lim_k  \int_{\{u_k>u_k^*(s)\}}  \diver\Big (a(|\nabla u_k|)\nabla u_k\Big)  \, dx \ge 0.
		 \end{equation*}
		 Hence, \eqref{eq:rearrangement div inequality} holds when $du^* / ds = 0$. Everywhere else $du^* / ds > 0$ so we can divide an recover the result.
\end{proof}

\begin{remark}
	An alternative proof of \eqref{eq:convergence int a(u_k)u_k} can be done by applying that the symmetrisation is non-expansive in $L^2 (\Omega)$.
\end{remark}

\subsubsection{Proof of \Cref{prop:Ph rearrangement}}
	We proceed as in \cite{Talenti1977} for the $\nabla$, and using standard inequalities for the rest. By \Cref{lem:PS} we have \eqref{eq:integrability of U_{ss}}. \\
	
	To check that the inequality of \eqref{eq:Ph star subsol} is satisfied, for $s \in [0,|\Omega_1|]$ we can integrate over the level set of $u_j$ 
\begin{eqnarray}
\label{p1}
 & & -\int_{u_j>u_j^*(s)} \diver\Big(a(|\nabla u_{{j}} |) 
\nabla u_{{j}} (x)\Big)\,dx
\\
\label{p2}
	& &	+\int_{u_j>u^*_j(s)} \frac{-u_{j+1} + 2u_j - u_{j-1}}{h^2} dx\\
\label{p3}
		& = & \int_{u_j>u^*_j(s)}  f_{{j}}(x) dx \, .
\end{eqnarray}
Notice that, due to \eqref{eq:regularity a(nabla u) nabla u} is, \eqref{p1} is well defined. Let us consider separately the three quantities which appear above.
	As regards to \eqref{p1}, we apply \Cref{lem:PS}, and hence, for a.e. $s \in \Omega_1^*$
	\begin{align}
		-\int_{u_j>u_j^*(s)} \left(\mbox{div}_x\left(a(|\nabla u_{{j}} |)\nabla u_j\right)\right)\,dx\label{p12} \ge   \kappa_n(s)\beta \left (- \kappa_n(s)\frac{\partial u_j^*}{\partial s}(s)\right ).
	\end{align}
	As regards to the term \eqref{p2}, it is a standard rearrangement inequality that
	\begin{equation*}
		\int_{u_j>u^*_j(s)} u_{k} \le \int_{0}^{s} u^*_k, \qquad \forall k 
	\end{equation*}
	and
	\begin{equation*}
		\int_{u_j>u^*_j(s)} u_{j} = \int_{0}^{s} u^*_j,
	\end{equation*}
	so that we get
	\begin{equation}\label{p22}
		\int_{u_j>u^*_j(s)} \frac{-u_{j+1} + 2u_j - u_{j-1}}{h^2} dx \ge \int_{0}^{s}  \frac{-u_{j+1}^* + 2u_j^* - u_{j-1}^*}{h^2} dx .
	\end{equation}
	Finally as regards \eqref{p3}, by a classical property of rearrangements we get
	\begin{equation}\label{p31}
	\int_{u_j>u^*_j(s)} f_{{j}} (x)\,dx \le \int_0^{s}f_{{j}} ^*(\sigma)\,d\sigma\,.
	\end{equation}
	Collecting \eqref{p12}-\eqref{p31} we get that the function $U_{j}$ is a weak solution of 
	$(\underline{\textrm{P}_h^*})$ with $\frac{d ^2 U_j}{d s ^2} \in L^\infty$. 
	This completes the proof for $U_j $. \\
	Analogously, the same arguments apply to the equation in \eqref{eq:Ph sharp}: since the solution $v_j$ equals $v_j^\#$, then all the inequalities in \eqref{p12}-\eqref{p31} hold as equalities.\qed

\subsection{Comparison principle for \eqref{eq:Ph star}. Mass comparison for \eqref{eq:Ph}}
\label{sec:comparison principle}

The aim of this section is to prove the following
\begin{proposition}
	\label{prop:Ph star comparison}
	Let $\mathbf U$ and $\mathbf V$ be as in \Cref{prop:Ph rearrangement}. Then $U_j \le V_j$ for all $j$. Hence
	\begin{equation}
		\int_0^s u_j^*  \le \int_0^s v_j^* \qquad \forall j,  \textrm{ a.e. } s \in [0,|\Omega_1|].
	\end{equation}
\end{proposition}
For the analysis of \eqref{eq:Ph star} we improve and close some open question raised in some previous literature concerning Hilbert spaces or reflexive Banach spaces \cite{Poffald+Reich1985,Reich+Shafrir1991,Benilan+Spiteri1987}. The keystone is to prove the so-called  $T$-accretivity in $L^\infty$ of some suitable operator. This is inspired in the proof of \cite[Theorem 1]{Diaz1991}.
Let us consider the operator 
\begin{equation*}
	A U =  \kappa_n(s) \beta \left (- \kappa_n(s) \,\frac{d^2 U}{d s^2}\right)\,,
\end{equation*}
defined in the domain
\begin{equation*}
	D(A) = \left\{  U \in L^\infty  (\Omega_1^*) :  \kappa_n(s) \,\frac{d^2 U}{d s^2} \in L^\infty (\Omega_1^*), \ \frac{dU}{ds} (|\Omega_1|) = 0,\ U(0) = 0    \right \}.
\end{equation*}

\subsubsection{The operator $A$ is $T$-accretive in $L^\infty$}

Let us prove that $A$ is $T$-accretive in $L^\infty$.
\begin{lemma}
	\label{lem:A is T accretive}
	Let $\beta$ be non-decreasing. Then,
	for all $U, V \in D(A)$ and $\lambda > 0$, we have that
	\begin{equation}
	\label{eq:A T-accrevity}
	\Big\| (U- V)_+ \Big\|_{L^\infty} \le \Bigg\| \Big(U-V + \lambda(AU - AV) \Big)_+ \Bigg\|_{L^\infty}.
	\end{equation}
\end{lemma}
\begin{proof}
	For the length of this section let $L = |\Omega_1|$.
	We check that
	\begin{equation*}
	A U =  \kappa_n(s) \beta \left( -  \kappa_n(s) U_{ss}  \right). 
	\end{equation*}
	There is an inverse operator
	\begin{equation*}
	\begin{dcases}
	A U = F \textrm{ in } (0,L) \\
	U_s (L) = U(0) = 0.
	\end{dcases}
	\end{equation*}
	We consider the even extension
	\begin{equation*}
	\widetilde U (s) = \begin{dcases}
	U(s) & s \in [0,L) , \\
	U(L-s) & s \in (L,2L).
	\end{dcases}
	\end{equation*}
	Since $U_s(0) = 0$, $\widetilde U(s)$ is a solution of
	\begin{equation*}
	\begin{dcases}
	\widetilde A \widetilde U = \widetilde F \textrm{ in } (0,2L), \\
	\widetilde U (0) = \widetilde U(2L) = 0,
	\end{dcases} 
	\end{equation*}
	where
	\begin{equation*}
	\widetilde A \widetilde U =  \widetilde {\kappa_n(s)} \beta ( - \widetilde  {\kappa_n(s)} \widetilde U_{ss} ) ,
	\end{equation*}
	and $\widetilde  {\kappa_n(s)}$ and $\widetilde F$ are the even extensions of $ \kappa_n(s)$ and $F$. 
	If $U \in D(A)$ then $ \kappa_n(s) U_{ss} \in L^\infty (0,L)$ and so $A U \in L^\infty (0,L)$. 
	Furthermore, since $ \kappa_n(s)^{-1} \in L^1 (\Omega)$ then $ U_{ss} =   \kappa_n(s)^{-1}  \kappa_n(s)  U_{ss} \in L^1 (0,L)$. Since $\kappa_n (s) > 0$ in $(0,L)$ we have that $\widetilde U_{ss} (s) \ge 0 $ if and only $\widetilde A \widetilde U (s) \ge 0$.	Finally, notice that $ \kappa_n(s)$ only vanishes at $0$, which does not affect the a.e. interior equalities. \\

	Suppose \eqref{eq:A T-accrevity} does not hold. Then, for some $U,V$ and $\lambda >0$ there exists $\mu > 0$ small such that 
	\begin{equation*}
	\| (U- V)_+ \|_{L^\infty} - \mu >   \| (U-V + \lambda(AU - AV))_+ \|_{L^\infty} \ge 0.
	\end{equation*}
	Thus $\| (U-V)_+ \|_{L^\infty} > \mu > 0$. The same holds for the even extension
	\begin{equation*}
	\| (\widetilde U- \widetilde V)_+ \|_{L^\infty} - \mu >   \| (\widetilde U-\widetilde V + \lambda(\widetilde A\widetilde U - \widetilde A\widetilde V))_+ \|_{L^\infty} \ge 0.
	\end{equation*}

	Define the closed set of positive measure
	\begin{align*}
	\Omega_+ &= \{ s \in (-L,L) :  (\widetilde U(s) - \widetilde V(s))_+ \ge - \mu +  \| (\widetilde U - \widetilde V)_+ \|_{L^\infty} \} .
	\end{align*}
	Notice that, in $\Omega_+$ we have $\widetilde U-\widetilde V > 0$. In particular, $\Omega_+ \Subset (0,2L)$. This set is selected so that
	\begin{equation*}
	\| (\widetilde U-\widetilde V)_+ \|_{L^\infty} = \max_{\Omega_+}  (\widetilde U-\widetilde V)_+   = \max_{\Omega_+} ( \widetilde U - \widetilde V ).
	\end{equation*}
	In $\Omega_+$ we have that
	\begin{align*}
	\widetilde U(s) - \widetilde V(s) &=  (\widetilde U(s) - \widetilde V(s))_+ \ge  -\mu + \| (\widetilde U-\widetilde V)_+ \|_{L^\infty} \\
	&>  \| (\widetilde U - \widetilde V+ \lambda(\widetilde A\widetilde U - \widetilde A\widetilde V))_+ \|_{L^\infty} \\
	&\ge \Big (\widetilde U(s) - \widetilde V(s) +  \lambda(\widetilde A\widetilde U(s) - \widetilde A\widetilde V(s)) \Big )_+ \\
	&\ge \widetilde U(s) - \widetilde V(s) +  \lambda(\widetilde A\widetilde U(s) - \widetilde A\widetilde V(s)).
	\end{align*}
	Therefore
	\begin{equation*}
	\widetilde A\widetilde U  < \widetilde A\widetilde V \textrm{ a.e. in } \Omega_+.
	\end{equation*}
	Then, $-\widetilde U_{ss} \le -\widetilde V_{ss}$ a.e. in $ \Omega_+$. Since $AU, AV \in L^\infty(0,L)$ we have that
	\begin{equation*}
	0 \ge -( \widetilde U - \widetilde V )_{ss} \in L^1 (\Omega_+).
	\end{equation*}
	Due to the maximum principle and the continuity of $\widetilde U-\widetilde V$, the maximum of $\widetilde U-\widetilde V$ is attained in $ \partial \Omega_+$. Therefore 
	\begin{equation*}
	- \mu + \| (\widetilde U -\widetilde V)_+ \|_{L^\infty} = \max_{\partial \Omega_+} (\widetilde U  - \widetilde V)  = \max_{\Omega_+} (\widetilde U  - \widetilde V) = \max_{\Omega_+}  (\widetilde U  -\widetilde V)_+ = \| (\widetilde U  -\widetilde V)_+ \|_{L^\infty}.
	\end{equation*}
	This is a contradiction. The proof is complete.
\end{proof}

\subsubsection{Proof of \Cref{prop:Ph star comparison}}
\label{sec:proof of comparison discrete}
Due to \eqref{eq:Ph star subsol} and \eqref{eq:Ph star} we have
\begin{equation*}
	\frac{h^2} 2 (A U_j - A V_j) + (U_j - V_j) \le \frac{1}{2} (U_{j+1} - V_{j+1}) + \frac{1}{2} (U_{j-1} - V_{j-1}). 
\end{equation*}
Applying \Cref{lem:A is T accretive}
\begin{equation*}
	\| (U_j - V_j)_+ \|_{L^\infty} \le  \frac{1}{2} \|(U_{j+1} - V_{j+1})_+\|_{L^\infty} + \frac{1}{2} \|(U_{j-1} - V_{j-1})_+\|_{L^\infty}.
\end{equation*}
We can rewrite this as
\begin{equation}
\label{ineqcrucial}
	\begin{pmatrix}
		2  & -1         \\
		-1 & 2         & -1         \\
		   &  \ddots & \ddots & \ddots  \\
		   &            & -1       & 2        & -1 \\
		   &            &           & -1       & 2
	\end{pmatrix} \begin{pmatrix}
		\| (U_1 - V_1)_+ \|_{L^\infty} \\
		\vdots \\
		\| (U_N - V_N)_+ \|_{L^\infty}
	\end{pmatrix} \le \mathbf 0
\end{equation}
where the inequality holds coordinate by coordinate. Let us call the matrix $D_2$ and denote the vector by $x$ in (\ref{ineqcrucial}). Notice that the vector components are non-negative.
We have the Cholesky decomposition
\begin{equation*}
	D_2 = C^t C, \qquad \text{where } C= \begin{pmatrix}
		1 & -1 \\
		0 & 1 & -1 \\
		& \ddots & \ddots & \ddots \\
		& & 0 & 1 & -1 \\
		& & & 0 & 1 
	\end{pmatrix}.
\end{equation*}
Multiplying (\ref{ineqcrucial}) by $x$, we obtain
\begin{equation*}
0\le 	\left \| C x \right \|_{2}^2 
	= x^t C^t C x = x^t A x \le 0.
\end{equation*}
Therefore $\| (U_j - V_j)_+\|_{L^\infty} = 0$. Hence $U_j \le V_j$. \qed 
 
\begin{remark}
	Notice that, in the proof, we use that the coefficient of the discrete Laplacian are non-positive outside the diagonal. More involved arguments can be applied overcoming this issue.
\end{remark}

\subsection{Convergence of the solutions of \eqref{eq:Ph} to the solution of \eqref{eq:P} as $h \to 0$}
\label{sec:convergence Ph to P}

From now on we use again the notation $\nabla_x$. Our aim is prove that we can pass to the limit $u^h \to u$ at least in $L^1 (\Omega_1 \times \Omega_2)$. This will be sufficient to show that the comparison of masses is preserved.\\

We make use of the \emph{floor function}: $$\lfloor z \rfloor = \min \{  k \in \mathbb Z : k \ge z  \}$$
i.e. $\lfloor z \rfloor =k $ means $k \le z < k +1$.
\begin{theorem}
	\label{thm:h to 0}
	Let $f \in \mathcal C^ \infty_c({\Omega_1 \times \Omega_2})$ and $\beta$ satisfy \eqref{eq:beta smooth}. Let $u$ denote the solution of \eqref{eq:P}. \\
	Let $N \in \mathbb N$ and let $h = 1/(N+1)$, and consider the constant interpolation
	\begin{align*}
	f^h (x,y) = f (x,jh) , \qquad j = \left \lfloor \frac{y}{h} \right \rfloor.
	\end{align*}
	Define $\mathbf u^h = (u_j^h)$ the unique solution of \eqref{eq:Ph} with data $\mathbf f = (f_j)$ and let us consider the linear interpolation
	\begin{align*}
	u^h (x,y) = u^h_j (x) + \frac{u^h_{j+1} (x) - u^h_j (x)}{h} (y - jh), \qquad j = \left \lfloor \frac{y}{h} \right \rfloor.
	\end{align*}
	Then
	\begin{enumerate}
		\item $u^h$ is a bounded sequence in $ H_0^1 (\Omega)$. 
		
		\item $u^h \rightharpoonup u$ weakly in $ H_0^1 (\Omega)$ as $h = \frac 1 {N+1} \to 0$.
	\end{enumerate}
	
\end{theorem}

\begin{remark}
	\label{rem:Minty's trick}
	We will apply the old trick {of} Minty \cite{Minty1962} (see also \cite[\S 5.1.3]{evans1988weak+convergence+methods} and \cite{Evans1980}): if $A$ is a monotone operator and $Au = f$, then for all test functions $\varphi$ we have
	$
	0 \le (Au - A\varphi, u - \varphi ) = (f - A\varphi, u - \varphi)
	$
	hence $$(A\varphi, \varphi - u) \ge (f, \varphi - u).$$
	One then recovers the equation by letting $\varphi = u + \lambda \psi$, so $\lambda (A(u + \lambda \psi), \psi) \ge \lambda (f, \psi)$. As $\lambda \to 0^+$ one has $(Au,\psi) \ge (f,\psi)$, while as $\lambda \to 0^-$ one has $(Au, \psi) \le (f, \psi)$. Hence $(Au, \psi) = (f,\psi)$, or $Au = f$. In particular, this trick applies if
	\begin{equation*}
		(A u, v) = \int_\Omega E (\nabla u) \nabla v
	\end{equation*} 
	and $E: \mathbb R^n \to \mathbb R^n$ is such that $(E(\xi) - E(\eta)) \cdot (\xi - \eta) \ge 0$ for all $\xi, \eta \in \mathbb R^n$. This is very advantageous when passing to the limit in formulations of type $(Au_k, v) = (f_k, v)$. It is in general difficult to pass to the limit under the nonlinearity, whereas in $(Av, v - u_k) \ge (f_k, v - u_k)$ the nonlinear term $Av$ is fixed and if $u_k$ has a weak limit and $f_k$ has a strong limit we can proceed. As we will see in \Cref{sec:proof of comparison general}, this is even more advantageous if we also have a sequence of operators $A_k$. 
\end{remark}

\begin{remark}
	We recall that the hypothesis \eqref{eq:beta smooth} make the problem elliptic of type $p = 2$ in $x$ and $y$, hence the natural space is $H_0^1 (\Omega)$.
\end{remark}

\begin{proof}[Proof of \Cref{thm:h to 0}]
	Let us check that $u^h$ is a bounded sequence in $H_0^1 (\Omega)$. We compute
	\begin{align*}
	\int_{\Omega_2}  \int_{\Omega_1} |\nabla_x u^h (x,y)|^2 dx  dy  &= \sum_{j=1}^N \int_{jh}^{(j+1)h} \int_{\Omega_1} \left|\nabla_x \left( u^h_j (x) + \frac{u^h_{j+1} (x) - u^h_j (x)}{h} (y - jh) \right) \right |^2 dx \\
	& \le C h \sum_{j=1}^N \int_{\Omega_1} \left|\nabla_x u^h_j (x) \right |^2 dx .
	\end{align*}
	On the other hand
	\begin{equation*}
	\frac{\partial u^h}{\partial y} (x,y) = \frac{u^h_{j+1} (x) - u^h_j (x)}{h}  \qquad j = \left \lfloor \frac{y}{h} \right \rfloor.
	\end{equation*}
	From \eqref{eq:Ph weak vector} we deduce that
	\begin{equation*}
		\int_{\Omega_1} \sum_{j=1}^N a(|\nabla u^h_j|) |\nabla u^h_j|^2 dx
		+ \int_{\Omega_1}  \sum_{j=0}^N \left( \frac{\partial u^h}{\partial y}  \right)^2 dx
		= \int_{\Omega_1} \sum_{j=1}^N f_j^h  u_j^hdx,
	\end{equation*}
	Since $a \ge \ee$ and the previous estimate, we deduce that
		\begin{equation*}
	\ee \int_{\Omega_1} \int_{\Omega_2} |\nabla u_j^h|^2 dydx
	+ \int_{\Omega_1}  \int_{\Omega_2} \left( \frac{\partial u^h}{\partial y}  \right)^2 dydx
	\le C \| f^h \|_{L^2 (\Omega)} \|u^h \|_{L^2 (\Omega)}.
	\end{equation*}
	Applying Poincaré's inequality we deduce the estimate.
	Thus, there exists $u \in H_0^1 (\Omega)$ such that
	\begin{equation*}
{u^h} \rightharpoonup u \text { in } H_0^1 (\Omega).
	\end{equation*}
	Let us check that $u$ is a solution of \eqref{eq:P}. Let $\varphi \in \mathcal C^4 (\overline \Omega) \cap \cC_0 (\overline \Omega) $.	Define
	\begin{equation*}
	 D_y^h \varphi (x,y) = \frac{\varphi (x, (j+1)h) - \varphi (x,jh)}{h}, \qquad j = \left \lfloor \frac{y}{h} \right \rfloor.
	 \end{equation*} 
	Going back to the weak formulation \eqref{eq:Ph weak vector} with $\varphi_j (x) = \varphi (x,jh)$ we have
	\begin{align}
	\int_{\Omega_1} &\sum_{j=1}^N E(\nabla_x u^h (x , jh)  ) \cdot \nabla_x \varphi (x, jh) dx \nonumber +  \int_{\Omega_1}  \sum_{j=1}^N \frac{\partial u^h}{\partial y} (x,jh)  D_y^h \varphi (x,jh) dx \nonumber \\ 
	&= \int_{\Omega_1} \sum_{j=1}^N f (x, jh)  \varphi (x, jh) dx.
	\label{eq:h to 0 1}
	\end{align}
	where, for convenience, we write $E(\xi) = a(|\xi) \xi$.
	Since the derivatives with respect to $y$ are piecewise constant
	\begin{align*}
		 \int_{\Omega_1} &h \sum_{j=1}^N E(\nabla_x u^h (x , jh)  ) \cdot \nabla_x \varphi (x, jh)  dx \nonumber 
		 +  \int_{\Omega_1}  \int_{\Omega_2 } \frac{\partial u^h}{\partial y} (x,y)  D_y^h \varphi  (x,y) dy \, dx \nonumber \\ 
		&= \int_{\Omega_1} h\sum_{j=1}^N f (x, jh)  \varphi (x, jh) dx.
	\end{align*}
	By the Taylor expansion, we know that
	\begin{equation*}
	\left \| \frac{\partial \varphi}{\partial y } - D_y^h u \right \|_{L^\infty} \le h	\left \| \frac{\partial^2 \varphi}{\partial y^2 } \right \|_{L^\infty}
	\end{equation*}
	Thus
	\begin{align*}
	\int_{\Omega_1} & h \sum_{j=1}^N E(\nabla_x u^h (x , jh)  ) \cdot \nabla_x \varphi (x, jh)  dx \nonumber
	+  \int_{\Omega_1}  \int_{\Omega_2 } \frac{\partial u^h}{\partial y}   \frac{\partial \varphi}{\partial y} dy \, dx 
	= \int_{\Omega_1} h \sum_{j=1}^N f (x, jh)  \varphi (x, jh) dx + R_1(h)
	\end{align*}
	where
	\begin{equation*}
		|R(h)| \le C h \left \| \frac{\partial u^h}{\partial y} (x,y) \right\|_{L^2} \left \| \frac{\partial^2 \varphi }{\partial y^2} (x,y) \right\|_{L^2}  \le C h.
	\end{equation*}
	Since $E$ is monotone we can apply Minty's trick (see \Cref{rem:Minty's trick}). We can write
		\begin{align*}
	\int_{\Omega_1} &h \sum_{j=1}^N E(\nabla_x \varphi (x , jh)  ) \cdot \nabla_x (\varphi (x,jh) - u^h (x,jh)) dx \nonumber
	+  \int_{\Omega_1}  \int_{\Omega_2 }  \frac{\partial \varphi}{\partial y } (x,y) \left(  \frac{\partial \varphi}{\partial y } (x,y)  -  \frac{\partial u^h}{\partial y } (x,y) \right) dy \, dx \nonumber \\ 
	&\ge \int_{\Omega_1} h \sum_{j=1}^N f (x, jh)  (\varphi (x, jh)  - u^h (x, jh))dx + R_1(h).
	\end{align*}
	Once the dangerous $u^h$ has been removed from the nonlinearity, we can apply the regularity of $\varphi$ and $f$ to deduce
	\begin{align*}
	\int_{\Omega_1} &\int_{\Omega_2} E(\nabla_x \varphi (x , y)  ) \cdot \nabla_x (\varphi (x,y) - u^h (x,y)) dx dy \nonumber 
	+  \int_{\Omega_1}  \int_{\Omega_2 }  \frac{\partial \varphi}{\partial y } (x,y) \left(  \frac{\partial \varphi}{\partial y } (x,y)  -  \frac{\partial u^h}{\partial y } (x,y) \right) dy \, dx \nonumber \\ 
	&\ge \int_{\Omega_1} \int_{\Omega_2} f (x, y)  (\varphi (x, y)  - u^h (x, y))dx + R_2(h).
	\end{align*}
	where $R_2 (h)$ also tends to zero. 
	So we can pass to the limit to deduce
	\begin{align}
	\int_{\Omega_1}& \int_{\Omega_2}E(\nabla_x \varphi (x , y)  )   \cdot \nabla_x (\varphi (x,y) - u (x,y)) dx dy \nonumber 
	+  \int_{\Omega_1}  \int_{\Omega_2 }  \frac{\partial \varphi}{\partial y } (x,y) \left(  \frac{\partial \varphi}{\partial y } (x,y)  -  \frac{\partial u}{\partial y } (x,y) \right) dy \, dx \nonumber \\ 
	&\ge \int_{\Omega_1} \int_{\Omega_2} f (x, y)  (\varphi (x, y)  - u (x, y))dx .
	\end{align}
	By density, we deduce that the formula also holds for $\varphi \in H_0^1 (\Omega)$. We take $\varphi = u + \lambda w$ with $w \in H_0^1 (\Omega)$. As $\lambda \to 0 ^\pm$ we recover the equation \eqref{eq:P}. This completes the proof.
\end{proof}

\subsection{Proof of \Cref{thm:main smooth}}
\label{sec:comparison in smooth case}

Let $f \in \mathcal C_c^\infty (\Omega)$. We construct the sequence given by \Cref{thm:h to 0}. We have that $u^h \to u$ in $L^1 (\Omega)$ as $h \to 0$. Analogously for $\Omega^\#$ we have that $v^h \to v$ in $L^1(\Omega^\#)$.
Therefore $(u^h)^* \to u^*$ and $(v^h)^* \to v^*$ in $L^1(\Omega_1^* \times \Omega_2)$.  Due to \Cref{prop:Ph star comparison} we have that
\begin{equation*}
	\int_0^s (u^h)^*(\sigma,y) d\sigma \le \int_0^s (v^h)^*(\sigma,y) d\sigma, \qquad \forall s \in (0,|\Omega_1|) , y \in \Omega_2.	
\end{equation*}
Passing to the limit we recover the result.

\section{Proof of the comparison result: \Cref{thm:main}}
\label{sec:proof of comparison general}

Since $A(t) = \beta(t) t$ is convex, by Moreau's theorem (see, e.g. \cite{Brezis1971MonotonicityMethods}), we define its Moreau-Yosida approximation which, for $t \ge 0$,
\begin{equation}
	A_\ee  (t) = \inf _{s \ge 0} \left(  \frac{1}{2\ee} | t-s|^2 + A(s)   \right).
\end{equation}
It is known that this function is differentiable and that its derivative is the Yosida approximation of the subdifferential of $A$, $\partial A$, of formula
\begin{equation*}
	A_\ee' (t) =  \frac{I - P_\ee (t) }{\ee},  \qquad   \text{where } P_\ee = ( I + \ee \partial A )^{-1}.
\end{equation*}
It is known that $A_\ee'$  is Lipschitz continuous of constant $1 / \ee$ and monotone increasing. Therefore $A_\ee$ is of class $C^1$. It is well known that, for all $t \ge 0$,
\begin{equation}
	\label{eq:Moreau-Yosida order}
	A(P_\ee (t)) \le A_\ee (t) \le A (t)
\end{equation}
and $J_\ee(t) \to t$ as $t \to 0$ and furthermore, $P_\ee $ converges upwards to $t$ and $A_\ee (t)$ converges upwards to $A(t)$.

Then, we define 
\begin{equation*}
	\beta_{\ee,\tau} (t) = \frac{A_\ee (t)}{t} + \tau t, \qquad a_{\ee,\tau} (t) = \frac{A_\ee (t)}{t^2} + \tau.
\end{equation*}
Since $A_\ee$ is convex and non-negative, $t A_{\ee} ' - A_\ee \ge 0$. On the other hand, since $A_\ee'(t)$ is Lipschitz and $A_\ee'(0) = 0$ we have that $A_\ee ' (t) \le t/\ee $.  Hence we check that $\tau \le \beta_{\ee,\tau}' (t) \le \frac 1 \ee + \tau $ and it is continuous. Clearly $\beta_{\ee,\tau} (t) t$ is convex. Hence $\beta_{\ee,\tau}$ satisfies \eqref{eq:beta smooth}. We also consider a sequence $f_\delta \in \cC^\infty_c (\Omega)$ such that
\begin{equation*}
	f_\delta \rightharpoonup f \text{ in } L^{\max\{2,p\}} (\Omega) \quad \text{as } \delta \to 0.
\end{equation*}

\paragraph{Step 1. For $f_\delta$ and $ \beta_{\ee,\tau}$} 

We construct the weak solutions $u_{\ee,\tau,\delta}$ and $v_{\ee,,\tau,\delta}$ corresponding to $a_{\ee,\tau}$ and $f_\delta$. Since $\beta_{\ee,\tau,\delta}(t) $ satisfies \eqref{eq:beta smooth} we can apply \Cref{thm:main smooth} (with $a \equiv a_{\ee,\tau}$ and $f \equiv f_\delta$) we have that
\begin{equation*}
	\int_0^s u^*_{\ee,\tau ,\delta} (\sigma,y) d\sigma \le \int_0^s v^*_{\ee, \tau ,\delta} (\sigma,y) d\sigma, \qquad \text{ a.e. } (s,y) \in \Omega_1^* \times \Omega_2.
\end{equation*}
Using $u_{\ee,\tau,\delta}$ as test function in the weak formulation we have that
\begin{align*}
	\int_{ \Omega } A_{\ee} (|\nabla_x u_{\ee,\tau,\delta}|) &+ \tau \int_{ \Omega } |\nabla_x u_{\ee,\tau,\delta}|^2 + \int_{ \Omega } |\nabla_y u_{\ee,\tau,\delta}|^2 \le C \| f_\delta \|_{L^{2}} \| u_{\ee,\tau,\delta} \|_{L^{2}}.
\end{align*}
Due to the Poincaré inequality in the variable $y$, we have that
\begin{equation}
	\label{eq:main proof 1}
		\int_{ \Omega } A_{\ee} (|\nabla_x u_{\ee,\tau,\delta}|)  + \tau \int_{ \Omega } |\nabla_x u_{\ee,\tau,\delta}|^2 + \int_{ \Omega } |\nabla_y u_{\ee,\tau,\delta}|^2 \le C \| f_\delta \|_{L^{2}}.
\end{equation}
and a similar bound for $v_{\ee,\tau,\delta}$. 
Here we are in the hypothesis \eqref{eq:Cianchi}, hence we can apply \cite[Lemma 4.2]{Cianchi2014} to show that $ E_{\ee,\tau} (\xi) = a_{\ee,\tau}(|\xi|) \xi $ is a monotone operator
\begin{equation*}
	\Big (a_{\ee,\tau} (|\xi|) \xi - a_{\ee,\tau} (|\eta|) \eta \Big ) \cdot (\xi - \eta) \ge 0 , \qquad \forall \xi, \eta \in \RR^n.
\end{equation*} 
Hence, we can apply Minty's trick (see \Cref{rem:Minty's trick}).
In our setting this reads 
\begin{equation}
	\label{eq:approx equation ee tau delta}
	\int_ \Omega a_{\ee,\tau} (|\nabla_x \varphi|) \nabla_x \varphi \cdot \nabla_x ( \varphi - u_{\ee,\tau,\delta} ) + \int_ \Omega \nabla_y \varphi   \cdot \nabla_y (\varphi - u_{\ee,\tau,\delta})  \ge  \int_ \Omega f (\varphi-u_{\ee,\tau,\delta}), \qquad \forall \varphi \in W^{1,\infty}_0 (\Omega).
\end{equation}

\paragraph{Step 2. Limit as $\delta \to 0$.} 
We pass to the limit as $\delta \to 0$. We have that $u_{\ee,\tau\delta} \in H_0^1 (\Omega)$ is uniformly bounded. Hence
\begin{align*}
	u_{\ee,\tau,\delta} &\to u_{\ee,\tau}   \quad \text{ in }L^2 (\Omega) , \\
	\nabla_y u_{\ee,\tau,\delta } &\rightharpoonup \nabla_y u_{\ee,\tau}  \quad \text{ in }L^2 (\Omega)^n , \\
	 \nabla_x u_{\ee,\delta }	 &\rightharpoonup \nabla_x u_{\ee,\tau} \quad  \text{ in }L^{2} (\Omega)^n.
\end{align*}
We can pass to the limit in \eqref{eq:approx equation ee tau delta} to recover
\begin{equation*}
	\int_ \Omega a_{\ee,\tau} (|\nabla_x \varphi|) \nabla_x \varphi \cdot \nabla_x (  \varphi - u_{\ee,\tau} ) + \int_ \Omega \nabla_y \varphi   \cdot \nabla_y ( \varphi - u_{\ee,\tau} )  \ge  \int_ \Omega f (\varphi -u_{\ee,\tau} ), \qquad \forall \varphi \in W^{1,\infty}_0 (\Omega).
\end{equation*}
	Since $A_\ee$ is convex, we can pass to the limit in \eqref{eq:main proof 1} to recover
	\begin{align}
	\label{eq:estimates in ee}
	 \int_{ \Omega } A_\ee   (|\nabla_x u_{\ee,\tau}|) &+ \tau \int_{ \Omega } |\nabla_x u_{\ee,\tau}|^2 + \int_{ \Omega } |\nabla_y u_{\ee,\tau}|^2 \le C.
	\end{align}
	Equivalent results apply to $v_{\ee,\tau,\delta}$ and $v_{\ee,\tau}$. Since $A_\ee (u_{\ee,\tau}) < +\infty$, it is immediate so see following \cite{ekeland+temam1999} that $u_{\ee,\tau}$ is a minimizer of the energy in $H_0^1 (\Omega)$
	\begin{equation}
	\label{eq:energy ee}
	J_{\ee,\tau}(w) = \int_{\Omega} \Big( B_\ee(|\nabla_x w|) + |\nabla_y w|^2 - f w \Big) dx\, dy + \frac{\tau}{2} \int_\Omega |\nabla_x w|^2 dx\, dy .
	\end{equation}
	where
	\begin{equation*}
		B_\ee (t) = \int_0^t \beta_\ee (\sigma) d\sigma.
	\end{equation*}
\paragraph{Step 3. Limit as $\varepsilon \to 0$.} 
	Due to the uniform bound \eqref{eq:estimates in ee}, up a to subsequence, we have, as $\ee \to 0$, that for every $\tau > 0$ there exists $u_\tau \in H_0^1(\Omega)$ such that, up to a subsequence,
\begin{align*}
	u_{\ee,\tau} &\to u_\tau  \quad  \text{ in }L^2 (\Omega) , \\
	\nabla u_{\ee,\tau } &\rightharpoonup \nabla u_\tau \quad  \text{ in }L^2 (\Omega)^n ,
\end{align*}
and the corresponding convergences for $v_\tau$.\\ 

Let us study the equation satisfied by $u_\tau$. Notice that, by pointwise convergence
\begin{equation*}
a_{\ee,\tau} (|\nabla_x \varphi|) \nabla_x \varphi \longrightarrow a_\tau (|\nabla_x \varphi|) \nabla_x \varphi \qquad \text{ a.e. } x \in \Omega.
\end{equation*}
Since the sequence is clearly uniformly bounded, we have that the convergence is strong in $L^q (\Omega)$ for any $1 \le q < +\infty$. This implies that $u$ is satisfies
\begin{equation}
	\label{eq:variational inequality delta}
	\int_ \Omega a_\tau (|\nabla_x \varphi|) \nabla_x \varphi \cdot (\nabla_x \varphi - \nabla_x u_\tau ) + \int_ \Omega \nabla_y \varphi   \cdot \nabla_y ( \varphi - u_\tau )  \ge  \int_ \Omega f (\varphi - u_\tau ), \qquad \forall \varphi \in W^{1,\infty}_0 (\Omega)
\end{equation}

Now it is not so simple to pass to the limit in the energy estimate \eqref{eq:estimates in ee}. We go back to the intrinsic energy.
Since $a_\tau(t) = A(t) / t^2 + \tau$, this is a variational formulation of the energy
\begin{equation*}
	J_\tau(w) = \begin{dcases} 
		J(w) + \frac{ \tau } 2  \int_\Omega |\nabla_x w |^2 & \text{if } B(|\nabla_x w|) \in L^1 (\Omega) \\
		+\infty & \text{otherwise}.
		\end{dcases} 
\end{equation*}
where $J$ is given by \eqref{eq:energy}. 

We now prove that $J_{\ee,\tau}$ $\Gamma$-converges in $H_0^1 (\Omega)$ to $J_{\tau}$ (see, e.g. \cite{DalMaso:1993}). This requires to prove two items:
\begin{enumerate}
	\item Given $w_\ee \to w$ in $H_0^1 (\Omega)$ show that $J_{\ee,\tau}(w) \le \liminf_\ee J_{\ee,\tau}(w_\ee)$. First, let us take a subsequence (still indexed as $\ee$) has as limit this $\liminf$. Due to strong convergence, a further subsequence has $\nabla w_\ee \to \nabla w$ converging a.e.. We can now apply Fatou's lemma to recover the result.
	
	\item For every $w \in H_0^1 (\Omega)$, show that that there exists $w_\ee \to w$ such that $J_\tau(w) \ge \limsup_\ee  J_{\ee,\tau} (w_{\ee,\tau})$. This is easy, we take $w_\ee = w$. Since $A_\ee $ converges upwards to $A$, so does $\beta_\ee$ to $\beta$ and so does $B_\ee $ to $B$. Therefore, by the Monotone Convergence Theorem we have the equality in the limits
	 $$J_\tau(w) = \lim_\ee  J_{\ee,\tau} (w).$$
	 Notice that this limit could be finite or not, depending on whether $B(|\nabla_x w|) \in L^1 (\Omega)$.
\end{enumerate}
Hence, the minimizers of $J_{\ee,\tau}$, i.e. $u_{\ee,\tau}$, converge to the minimizer of $J_\tau$. In particular, its already known limit $u_\tau$ is the minimizer of $J_\tau$. 
Hence, $u_\tau$ is a weak solution of
\begin{equation*}
	-\diver ( a(|\nabla_x  u_\tau|) \nabla_x  u_\tau ) - \tau \Delta_x u_\tau - \Delta_y u_\tau = f.
\end{equation*} 
Thus, we recover uniform bounds of $|\nabla_x u_\tau|^p$ and $|\nabla_y u_\tau|^2$.

\paragraph{Step 4. Limit as $\tau \to 0$.}  Due to the uniform bounds on the gradient, we have that there exists $u \in X^p (\Omega)$ such that, up to a subsequence,
\begin{align*}
u_{\tau} &\to u   \quad \text{ in }L^2 (\Omega) , \\
\nabla_y u_{\tau } &\rightharpoonup \nabla_y u  \quad \text{ in }L^2 (\Omega)^n , \\
\nabla_x u_{\tau }	 &\rightharpoonup \nabla_x u \quad  \text{ in }L^{p} (\Omega)^n.
\end{align*}
Since $a_\tau$ converges uniformly to $a$, we recover that $u$ satisfies
\begin{equation*}
	\int_ \Omega a (|\nabla_x \varphi|) \nabla_x \varphi \cdot (\nabla_x u - \nabla_x \varphi ) + \int_ \Omega \nabla_y \varphi   \cdot \nabla_y (u - \varphi)  \ge  \int_ \Omega f (u - \varphi), \qquad \forall \varphi \in W^{1,\infty}_0 (\Omega)
\end{equation*}
and the analogous for $v$. By Minty's trick, we show that $u$ is the unique weak solution of \eqref{eq:P}. 
Proceeding analogously for $v$ we have obtained
the solution of \eqref{eq:P sharp}. Since all limits can  been taken strongly in $L^1$, the mass comparison is preserved,
\begin{equation*}
\int_0^s u^* (\sigma,y) d\sigma \le \int_0^s v^* (\sigma,y) d\sigma, \qquad \text{ a.e. } (s,y) \in \Omega_1^* \times \Omega_2.
\end{equation*}
This completes the proof. \qed 
\section{Extensions, generalizations and open problems}
By some technical adaptations of the proof we will prove in an upcoming paper the following generalisations:

\begin{enumerate}
	\item The case of $\beta$ a multivalued maximal monotone graphs. Then, \eqref{eq:beta p estimates} becomes
	\begin{equation}
		C_1 (t^p - 1) \le t w ,\qquad \forall w \in \beta(t).
	\end{equation} 
	The approximation argument follows similarly to \cite[Theorem 2.15]{Barbu2010}. 
	
	\item Case of $\Omega_2 \in \mathbb R^{n_2}$ with $n_2 > 1$. The finite-differences approach still works, although with some modifications. The boundary nodes are no longer $j = 0, N$, and the structure of the matrix $D_2$ in \eqref{ineqcrucial} is more complicated. Some regularity assumptions on $\Omega$ will be required for the convergence of the finite difference scheme.
	
	\item Case of $\Omega$ not a product, but any general domain in $\mathbb R^{n_1+n_2}$. Then, one needs to replace $\Omega_1$ by the cuts $\Omega_y = \Omega \cap \{  (x, y) : x \in \mathbb R^n  \}$ and rearrange in each of them. This introduces some additional technical difficulties.

	\item Some more general operators in $y$ can be studied. For example, operators of the form $-\diver (A(y) \nabla_y u)$, under some assumptions on $A(y)$.
	
	\item The boundary condtion $u = 0$ in $\partial \Omega_1 \times \Omega_2$ is a known requirement for the rearrangement. However, one could set a Neumann boundary condition on $\Omega_1 \times \partial \Omega_2$. Our method can be adapted to this case by considering a discretisation of $\frac{\partial u}{\partial n}$ at the first and last nodes.
		
	\item One could add a zero order term, as in \cite{Diaz2016}.
\end{enumerate}

There are some alternative approaches that could prove succesful, and perhaps more direct, although they require stronger theory. They could be applied to more general operators in $y$.
\begin{enumerate} 
	\item   The structure \eqref{eq:Ph star subsol} shows that $U(s,y) = \int_0^s u(\sigma,y) d\sigma$ is such that $ \kappa_n(s) \partial^2 U / \partial s^2 \in L^\infty$ and $\nabla_y U \in L^2$ and a solution of 
	\begin{equation}
	\begin{dcases} 
		 \kappa_n(s) \beta \left (- \kappa_n(s) \,\frac{\partial^2 U}{\partial s^2}\right) -\Delta_y U \le F & \Omega_1^* \times \Omega_2  \\
		\frac{\partial U}{\partial s} (|\Omega_1|,y ) = 0, \quad U(0,y) = 0 , & y \in \Omega_2 \\
		U(s,y) = 0 & \Omega_1^* \times \partial \Omega_2.
	\end{dcases} 
	\end{equation}
{There is a comparison principle for viscosity solutions of this problem }(see \cite{crandall+ishii+lions1992users-guide-viscosity}).
{Then the difficulty is to show that $U$ is a viscosity solution.} 
We have avoided this difficulty through  the discretisation in $y$.
 For the linear case, the fact that weak regular solutions are viscosity solutions is a well-known fact (see \cite{Ishii1995}). Some recent works show that this is also the case for the $p$-Laplace operator \cite{Juutinen2003,Medina2019}
	
	\item 
	A natural replacement of the finite differences are the finite elements.
One {can} consider a mesh over an approximate domain $(\Omega_2)_h$, and a basis $(\varphi_j^h)$ over the mesh and writes an approximate solution
	\begin{equation*}
		u^h(x,y) = \sum_{j=1}^{N} u_j^h (x) \varphi_j^h(y), \qquad f_h(x,y) = \sum_{j=1}^N f_j^h(x) \varphi_j^h(y) \quad \textrm{ where } \quad f_j^h(x) = \int_{\Omega_2} f(x,y) \varphi_j^h (y) dy. 
	\end{equation*}
	The problem is written as a system of boundary value problems in $s$, a solution of which exists by minimization. 
{Via rearrangement one obtains} functions
	\begin{gather*}
		\int_0^s(u^h)^*(\sigma,y) d\sigma \le  \sum_{j=1}^{N} \int_0^s(u_j^h)^* (\sigma) \varphi_j^h(y) d\sigma\\
		U(s,y)  = \int_0^s (u^h)^*(\sigma,y) d\sigma = \sum_{j=1}^{N} \int_0^s (u_j^h)^*(\sigma,y) d\sigma\  \varphi_j(y) = \sum_{i=1}^N U_j(s) \varphi_j(y).
	\end{gather*}
	This should lead to a discrete problem in $U_j$ that still preserves the comparison principle. 
\end{enumerate}

\section*{Acknowledgements}
The research of J.I. D\'iaz and D. Gómez-Castro was partially supported by the
project ref. MTM2017-85449-P of the Ministerio de Ciencia, Innovaci\'on y Universidades --
Agencia Estatal de Investigaci\'on (Spain). 

\noindent The research of D. Gómez-Castro
was partially supported by grant PGC2018-098440-B-I00 from the Ministerio de
Ciencia, Innovaci\'on y Universidades -- Agencia Estatal de Investigaci\'on (Spain).

\noindent The research of A. Mercaldo was partially supported by Gruppo Nazionale per
l'Analisi Matematica, la Probabilit\`a e le loro Applicazioni (GNAMPA)
of the Istituto Nazionale di Alta Matematica (INdAM).

\bibliographystyle{abbrv}
\bibliography{bibliography.bib}

\end{document}